\documentclass{amsart}
\usepackage{amsmath,amsthm,amssymb,bm}
\usepackage{color}
\usepackage{mathtools}
\mathtoolsset{showonlyrefs=true} %\UTF{00E5}\UTF{0152}\UTF{0095}\UTF{00E7}\UTF{0094}\UTF{0161}\UTF{00E3}\UTF{0081}\UTF{0097}\UTF{00E3}\UTF{0081}\UTF{009F}\UTF{00E5}\UTF{0152}\UTF{008F}\UTF{00E3}?\UTF{00E3}\UTF{0081}\UTF{00BF}\UTF{00E7}\UTF{0095}\UTF{00AA}\UTF{00E5}\UTF{008F}\UTF{00B7}\UTF{00E3}\UTF{0082}\UTF{0092}\UTF{00E5}?\UTF{00E5}\UTF{008A}\UTF{009B}\UTF{00E3}\UTF{0081}\UTF{0099}\UTF{00E3}\UTF{0082}? mathtools

\usepackage{ulem}

\usepackage{tikz}
\usetikzlibrary{intersections,patterns}

\makeatletter
    
    \@addtoreset{equation}{section}
  \makeatother

\newtheorem{definition}{Definition}[section]
\newtheorem{proposition}[definition]{Proposition}
\newtheorem{theorem}[definition]{Theorem}
\newtheorem{lemma}[definition]{Lemma}

\newtheorem{remark}{Remark}[section]

\newcommand{\re}{\mathbb R} 

\def \ep{\varepsilon}

\DeclareMathOperator{\diver}{div}
\DeclareMathOperator{\rot}{rot}
\DeclareMathOperator{\supp}{supp}
\DeclareMathOperator{\Int}{Int}
\DeclareMathOperator{\Ext}{Ext}
\pagebreak

\title[Asymptotic behavior for elliptic equations in 2D exterior domain]{
Asymptotic behavior of solutions to elliptic equations in 2D exterior domains \\
} 
\author{Hideo Kozono, Yutaka Terasawa and Yuta Wakasugi}

\address[H. Kozono]{Department of Mathematics, Faculty of Science and Engineering,
Waseda University, Tokyo 169--8555, Japan, 
Mathematical Research Center for Co-creative Society, Tohoku University, 
Sendai 980-8578, Japan}
\email[H. Kozono]{kozono@waseda.jp, hideokozono@tohoku.ac.jp}
\address[Y. Terasawa]{Graduate School of Mathematics, Nagoya University,
Furocho Chikusaku Nagoya 464-8602, Japan}
\email[Y. Terasawa]{yutaka@math.nagoya-u.ac.jp}
\address[Y. Wakasugi]{Graduate School of Advanced Science and Engineering,
Hiroshima University,
Higashi-Hiroshima, 739-8527, Japan}
\email[Y. Wakasugi]{wakasugi@hiroshima-u.ac.jp}

\keywords{Second order elliptic equations; strong maximal principle; asymptotic behavior; level sets}

\begin{document}
\maketitle

%
%\begin{dedication}
\textit{Dedicated to Professor Toshiaki Hishida on the occasion of his 60th birthday}
%\end{dedication}

\begin{abstract}
The asymptotic behavior of solutions to second order elliptic equations in
two-dimensional
exterior domains is studied.
In particular, under the assumption that the solution belongs to
the Lorentz space $L^{p,q}$ or the weak Lebesgue space $L^{p,\infty}$
with certain conditions on the coefficients, we give a natural and the optimal sharp pointwise estimate of the solution at spatial infinity.
The proof is based on the level set approach of solutions introduced by Korobkov--Pileckas--Russo \cite{KoPiRu19},
in which the decay property of the solution to the vorticity equation of the two-dimensional Navier--Stokes equations was studied.
\end{abstract}
%\keywords{Second order elliptic equations; strong maximal principle; asymptotic behavior; level sets}
%\maketitle
\section{Introduction}
\footnote[0]{2010 Mathematics Subject Classification. 35J15; 35K10; 35B53}
%\footnote[0]{Dedicated to Prof.Toshiaki Hishida on the occasion of his 60th's birthday }

Let $\Omega \subset \mathbb{R}^2$ be an exterior domain with smooth boundary $\partial \Omega$.
We study the following second order elliptic equation in $\Omega$:
\begin{equation}\label{eq:elliptic}
	Lu := - \sum_{i,j=1}^2 \partial_i ( a_{ij}(x) \partial_j u ) + \sum_{j=1}^2 b_j(x) \partial_j u + c(x) u = 0.
\end{equation}
Here, the coefficients $a_{ij}, b_j, c$ are smooth functions,  and
$(a_{ij})_{i,j=1,2}$ is assumed to be uniformly elliptic.
%(the precise assumptions are given in \textbf{Assumption (C)} below).
%
The problem \eqref{eq:elliptic} is motivated from
the vorticity equation of two-dimensional stationary incompressible fluids
\begin{equation}\label{eq:vor}
	-\Delta \omega + v \cdot \nabla \omega = 0 \quad \text{in} \ \Omega.
\end{equation}
Here, $v(x) = (v^1(x), v^2(x))$ is the velocity vector and
$\omega(x) = \rot v(x)$.
The asymptotic behavior of the solution $\omega$
under the condition of finite Dirichlet integral
$\nabla v \in L^2(\Omega)$
was studied by
Gilbarg--Weinberger \cite{GiWe78},
Amick \cite{Am88},
and Korobkov--Pileckas--Russo \cite{KoPiRu19, KoPiRu20}.
They proved that $\omega$ satisfies
\[
	\omega(x) = o(|x|^{-3/4})  \quad \text{as} \ |x| \to \infty.
\]
Recently, the authors \cite{KoTeWa22} obtained the asymptotic behavior
\[
	\omega(x) = o(|x|^{-(1/p+1/p^2)}) \quad \text{as} \ |x| \to \infty
\]
under the generalized finite Dirichlet condition $\nabla v \in L^p(\Omega)$ with some $p \in (2,\infty)$.
For the study of the asymptotic behavior of the solution to the vorticity equation \eqref{eq:vor},
it is crucial to treat the velocity $v(x)$ as a given coefficient and
to clarify how the decay property of $v(x)$ at spatial infinity influences that of $\omega(x)$.
Such an observation naturally indicates the following question:
\textit{for general second-order elliptic equations \eqref{eq:elliptic},
under what conditions on the coefficients can we obtain the pointwise decay properties at spatial infinity 
for the solution belonging to $L^p(\Omega)$ ?}
\par
To this question, our previous result \cite{KoTeWa21} showed that,
under the assumptions that the coefficients satisfy
\[ 
	|a_{ij}(x)| = O(|x|^{\alpha}),\quad |b_j(x)| = O(|x|^{\beta}) \ \text{as}\ |x|\to \infty,
	\quad c(x) \ge 0,
\]
with some $\alpha \in [0,2]$ and $\beta \le 1$ and either
$\diver (b_1, b_2) \le 2 c(x)$ or $|\diver (b_1, b_2) (x)| = O( |x|^{\beta-1} )$,
the asymptotic behavior of the smooth solution $u \in L^p (\Omega)$ with some $p \in [2,\infty)$
is given by
\[
	|u(x)| = o (|x|^{-\frac{1}{p}\left( 1+ \frac{\gamma}{2} \right)}) \ \text{as} \ |x|\to \infty
\]
with $\gamma = \min \{ 1 - \beta, 2 - \alpha \}$.
In particular, when $\alpha = 0$ and $\beta \le -1$, we have
\[
	|u(x)| = o(|x|^{-\frac{2}{p}} ) \ \text{as} \ |x| \to \infty,
\]
which seems natural and almost optimal in view of the assumption $u \in L^p(\Omega)$.
Moreover, as a corollary, we have the following Liouville-type result:
let $\Omega = \mathbb{R}^2$ and let $u$ be a classical solution to \eqref{eq:elliptic} satisfying $u \in L^p(\mathbb{R}^2)$
with some $p \in [2,\infty)$, then $u \equiv 0$ in $\mathbb{R}^2$.
The analysis of \cite{KoTeWa21} is based on the classical result by Gilbarg--Weinberger \cite{GiWe78}.
The novelty of this method is to apply the energy estimate, the integral mean value theorem for the radial variable,
the fundamental theorem of calculus for the angular variable, and the maximum principle.
We also refer readers to \cite{KoTeWa23}, \cite{SeSiSvZl12} and the references therein
for the asymptotic behavior and Liouville-type theorems
of the 3D Navier-Stokes equations and elliptic equations in general dimensions.
\par
In this paper, we study a similar almost optimal estimate including the cases $p \in [1,2)$
under different conditions on the coefficients by means of another approach.
To state our main result, we impose the following assumptions 
on the coefficients of the differential operator $L$ in (\ref{eq:elliptic}).

\noindent
\textbf{Assumption (C)}

$\{a_{ij}\}_{i, j =1, 2}\in C^1(\Omega)\cap L^\infty(\Omega)$, $\mathbf{b} =(b_1,b_2)\in C^1(\Omega)\cap C(\overline{\Omega})$ and $c \in C({\Omega})$ satisfy
\begin{itemize}
\item[\textbf{(C1)}] There exists some constant $\lambda > 0$ such that 
\[
	\sum_{i,j=1}^2 a_{ij}(x) \xi_i \xi_j \ge \lambda |\xi|^2
	\quad
	\mbox{for all $x \in \Omega$ and $\xi \in \mathbb{R}^2$; }
\]
\item[\textbf{(C2)}] $\nabla a_{ij}(x) = O(|x|^{-1})$ and $\mathbf{b}(x) = O(|x|^{-1})$ as $|x| \to \infty$; 
\item[\textbf{(C3)}] $c(x) \ge 0$ for all $x \in \Omega$; 
\item[\textbf{(C4)}] $(\diver \mathbf{b} - c)_- \in L^1(\Omega)$, 
where $f_{-}\equiv\max\{0, -f\}$.
\end{itemize}
\par
Furthermore, we impose the following assumptions on the solution $u$.
\par
\noindent
\textbf{Assumption (S)}
\begin{itemize}
\item[\textbf{(S1)}] $u \in C^{2}(\Omega) \cap C(\overline{\Omega})$; 
\item[\textbf{(S2)}] $\lim_{|x|\to \infty} u(x) = 0$; 
\item[\textbf{(S3)}] $\left. u \right|_{\partial \Omega} \ge 0$.
\end{itemize}
\par
\bigskip
Our main result now reads: 

%%%%%%%%%%%%%%%%%%

\begin{theorem}\label{thm1} 
Let the coefficients $(a_{ij})_{i, j=1,2}$, $\mathbf{b}=(b_1, b_2)$ and $c$ satisfy 
\textup{\textbf{Assumption}} \textup{\textbf{(C)}}. 
Suppose that $u$ is the solution of (\ref{eq:elliptic}) satisfying  
\textup{\textbf{Assumption}} \textup{\textbf{(S)}}. 
Then we have the following decay property of $u$:
\begin{itemize}
\item[(i)]
If $u \in L^{p,q}(\Omega)$ with some $p \in [1,\infty)$ and $q \in [1,\infty)$,
then $u$ has the pointwise decay
$u(x) = o(|x|^{-2/p})$ as $|x| \to \infty$.
\item[(ii)]
If $u \in L^{p,\infty}(\Omega)$ with some $p \in [1,\infty)$,
then $u$ has the pointwise decay $u(x)= O(|x|^{-2/p})$ as $|x| \to \infty$.
\end{itemize}
\end{theorem}

\begin{remark}\label{rem:1.1}{\rm (i)}
The decay rate $|x|^{-2/p}$ are almost optimal in view of the assumptions
$u \in L^{p,q}(\Omega)$ and $u \in L^{p,\infty}(\Omega)$ in (i) and (ii), respectively.
Compared with the previous result \cite{KoTeWa21}, we refine the range of 
integral exponent $p$ denoting the decay of the solution $u$ at spatial infinity 
from $L^p(\Omega)$ with $p \in [2,\infty)$ to
$L^{p,q}(\Omega)$ with $p\in [1,\infty), q \in [1,\infty]$.
Concerning the assumption on the coefficients,
{\rm {\textbf{Assumption (C)}}} requires that they need to be bounded, which is stronger than \cite{KoTeWa21}.
On the other hand, the conditions on $\diver \mathbf{b}$ in {\rm \textbf{(C4)}} 
is assumed to have an integral form,
which is weaker than \cite{KoTeWa21}.
We also remark that in the case $\diver \mathbf{b} = 0$, {\rm \textbf{(C4)}} is replaced by the assumption that $c \in L^1(\Omega)$.
\par
{\rm (ii)} Indeed, the following example shows that our decay rate such as (i) and (ii) in Theorem \ref{thm1} is optimal.  Assume that $\Omega$ is an exterior domain 
with
$0 \in (\bar{\Omega})^c = \re^2\setminus \bar{\Omega}$.
Consider the case $a_{ij}(x) = \delta_{ij}$, $i, j =1,2$, 
$\mathbf{b}(x) = \frac 2p |x|^{-2}(x_1, x_2)$ and $c(x)=0$ for $1 \le p  < \infty$. 
Obviously, for such coefficients  Assumption (C) is fulfilled. 
Let $u(x) = |x|^{-\frac 2p}$.  It is easy to see that $u$ is a solution of (\ref{eq:elliptic})  
satisfying Assumption (S). Since this $u$ satisfies that $u \in L^{p, \infty}(\Omega)$ and 
$u \in  L^{r, q}(\Omega)$ for all $p < r < \infty$ and $1 \le q \le \infty$, we see that 
our decay rate as in Theorem \ref{thm1} is optimal under the general setting of 
Asumptions (C) and (S).   
\end{remark}

Although the idea of our proof is based on the argument by Korobkov--Pileckas--Russo \cite{KoPiRu19}, 
our method seems so refined as to be applicable to generalized elliptic equations 
in two-dimensional exterior domains.
For the vorticity equation \eqref{eq:vor} they proved that 
$\omega(x) = o(|x|^{-1})$
under the condition of the finite Dirichlet integral by making use of the fact 
that the level sets of $\omega$ separate infinity from the origin. 
In contrast to their method, the first key point of our proof is to show that 
the level sets of $u$ except for a set of values of measure zero
consist of a family of disjoint closed curves containing the obstacle 
$\Omega^c=\re^2\setminus \Omega$ and of other connected components within such closed curves.
It should be noted that they \cite{KoPiRu19} do not need to handle 
the latter harmful connected components of the level sets 
because the vorticity equation \eqref{eq:vor} has such a simple structure 
as $\omega$ attains neither maximum nor minimum in any interior sub-region of $\Omega$.  
On the other hand, our elliptic operator $L$ in \eqref{eq:elliptic} has a general structure 
with the lower order term such as Assumption (C3) so that only non-negative maximum and 
non-positive minimum of $u$ cannot be attained. This is a crucial difference of dealing with 
the level sets of solutions between \cite{KoPiRu19} and our case.  
\par       
The second one is to estimate the integral of the gradient of the solution on the level set curve,
from which and the coarea formula we obtain a bound of the length of level set curve. 
By using this bound and the spetial property of the two-dimensional geometry, 
we are able to show the desired pointwise estimate of the solution.
In \cite{KoPiRu19}, they made use of such a priori bounds as
$\omega \in L^2(\Omega)$ and $\nabla \omega \in L^2(\Omega)$,
which had been already proved by the pioneer work by Gilbarg--Weinberger  \cite{GiWe78}.  
In this paper, we remove such an assumption on a priori estimates and 
successfully modify their argument by using a cut-off method.
\par
\vskip2mm
Finally in this section, we introduce the notations used throughout this paper.
The letter $C$ indicates a generic constant which may change from line to line.
Sometimes we use the notation $C (\ast,\ldots,\ast)$ for a constant depending only on the quantities in the parentheses.
For $R > 0$, we denote
$B_R:= \{ x \in \mathbb{R}^2 ; \, |x| < R \}$.
For a function $f=f(x)$,
$f_+ := \max \{ f, 0 \}$ and $f_- := f_+ - f$
are the positive and negative parts of $f$, respectively.
For a Lebesgue measurable set $E$ in $\mathbb{R}^2$,
$|E|$ stands for the Lebesgue measure of $E$.
Let $\mathcal{H}^1(F)$ be the $1$-dimensional Hausdorff measure of $F$.
\par
For a simple closed curve $\gamma$ in $\mathbb{R}^2$,
the Jordan curve theorem implies that $\gamma$ divides $\mathbb{R}^2$ into
the bounded interior region and the unbounded exterior region.
We call these regions $\Int{\gamma}$ and $\Ext{\gamma}$, respectively.
\par
For $1 \le p < \infty$ and $1 \le q \le \infty$,
$L^{p,q}(\Omega)$ denotes the Lorentz space defined by
\[
	L^{p,q}(\Omega) = \left\{ f : \Omega \to \mathbb{R} ;\, \| f \|_{L^{p,q}} < \infty \right\}
\]
with
\[
	\| f \|_{L^{p,q}} = 
				\begin{dcases}
				\left( p\int_0^{\infty} t^q \left( | \{ x \in \Omega ;\, |f(x)| \ge t \} | \right)^{q/p} \,\frac{dt}{t} \right)^{1/q}
				&(1 \le q < \infty),\\
				\sup_{t>0} t | \{ x \in \Omega ;\, |f(x)| \ge t \} |^{1/p}
				&(q= \infty).
				\end{dcases}
\]

\section{Proof of Theorem \ref{thm1}}
\subsection{Geometry of level sets} 
Since our method is based on the growth rate of the level set of solutions $u$ to (\ref{eq:elliptic}),
we need to investigate its geometric properties. 
First, if $u \equiv 0$, the claim of the theorem obviously holds.
Thus, noting the condition \textbf{(S2)}, we may suppose that $u$ is not identically a constant.

By the assumption $c \ge 0$,
the operator $L$ has the strong maximum principle
(see e.g., \cite[Theorem 3.5]{GiTr}).
Let $u$ be a smooth solution to \eqref{eq:elliptic} satisfying \textbf{Assumption (S)}.
Then, by the strong maximum principle, $u$ takes neither a non-negative maximum nor non-positive minimum
in the interior of $\Omega$.
Therefore, if $u$ attains a negative value, then it contradicts the condition \textbf{(S2)}.
Thus, $u$ must be nonnegative in $\Omega$.
Then, using the strong maximum principle again, we conclude that $u$ is positive in the interior of $\Omega$.

Let $B_R$ be an open ball with radius $R>0$ satisfying 
$\Omega^c = \re^2\setminus \Omega\subset B_R$.
Then, we have
\[
	t_* := \min_{x \in \partial B_R} u(x) > 0.
\]

Next, the Morse--Sard theorem (see e.g., \cite{Mi65} for a simple proof)
implies that almost every $t \in (0,t_*)$ is a regular value of $u$.
Let us define $\mathcal{I}$ by 
\begin{equation}
\label{eq:I}
	\mathcal{I} := (0,t_*) \setminus \mathrm{Cr}(u),
\end{equation}
where $\mathrm{Cr}(u) :=  \{u(x); \,\nabla u(x) = \mathbf{0},\, x \in B_R^c=\re^2\setminus  B_R\}$ denotes the set of critical values of $u$.   
Thus, it follows from the implicit function theorem that for any $t\in \mathcal{I}$,
the level set $u^{-1}(t)$ is a union of smooth simple closed curves.
Moreover, we have the following proposition on the geometric property of $u^{-1}(t)$. 
%%%%%%%%%%%%%%%%%%%%%%%%%%%%
\begin{proposition}\label{pr:2.1} 
Let $u$ be a solution to \eqref{eq:elliptic} satisfying \textbf{Assumption (S)}.
For every $t \in \mathcal{I}$,
$u^{-1}(t)$ has a unique connected component $\gamma(t)$
such that $B_R \subset \Int \gamma(t)$.
Further, there exists no connected component of $u^{-1}(t)$ which lies
in $\Ext \gamma(t)$.
\end{proposition} 
%%%%%%%%%%%%%%%%%%%%%%%%%
\begin{proof}   
We first show that there exists at least one connected component,   
say $\gamma(t)$, of $u^{-1}(t)$, which contains $B_R$ inside.  
Suppose the contrary, which means that all connected components of $u^{-1}(t)$ 
lie outside of $B_R$.  Then there are some point $x_\ast \in \partial B_R$ 
and a continuous curve $\{\mu(s)\in \re^2; 0 \le s < \infty \}$ with 
$\mu(0)= x_\ast$ and $\lim_{s\to\infty}\mu(s) = \infty$ such that 
for all $0 \le s < \infty$, 
$\mu(s)$ does not intersect any connected component of $u^{-1}(t)$.  
Since $u \in C(\overline {\Omega})$ and since  
$u(\mu(0)) = u(x_\ast) \ge t_\ast > t >0$ with $\lim_{s\to\infty}u(\mu(s))=0$, 
it follows from the intermediate value theorem that there is some $0 < s_\ast < \infty$ such that $u(\mu(s_\ast)) = t$, which implies that 
$\mu(s_\ast) \in u^{-1}(t)$.  This causes a contradiction. 
\par 
Next, we show the uniqueness of $\gamma(t)$ above and
that there is no connected component of $u^{-1}(t)$ which lies in $\Ext \gamma(t)$.
Since $\lim_{|x|\to\infty}u(x)=0$, we see again by the maximum principle that 
$u$ does not attain the non-negative maximum in
$\Ext \gamma(t)$,
which yields that $u(x) < t$ in
$\Ext \gamma(t)$.
From this, we conclude the uniqueness of the connected component of $u^{-1}(t)$ containing $B_R$ inside.
In fact, besides $\gamma(t)$ above,
if there exists another connected component $\tilde{\gamma}(t)$ of $u^{-1}(t)$ satisfying
$B_R \subset \Int \tilde{\gamma}(t)$,
then we have either $\tilde{\gamma}(t) \subset \Ext \gamma(t)$ or $\gamma(t) \subset \Ext \tilde{\gamma}(t)$.
However, this is impossible because $u(x) < t$ in $\Ext \gamma(t)$ and $\Ext \tilde{\gamma}(t)$.
The same argument also shows that no connected component of $u^{-1}(t)$ can exist in $\Ext{\gamma}(t)$.
This proves Proposition.
\end{proof}

The simple closed curve $\gamma(t)$, whose existence is guaranteed by Proposition \ref{pr:2.1},
divides $B_R^c$ into two parts $\Int \gamma(t) \cap B_R^c$ and $\Ext \gamma(t)$.
In what follows, we also denote $\Ext \gamma(t)$ by $\Omega_t$.
Then, by the assumption $\lim_{|x|\to \infty} u(x) = 0$,
we have
$\Omega_t \subset \Omega_s$ if $t, s \in \mathcal{I}$ and $t<s$.
Indeed, otherwise we have $\Omega_s \subset \Omega_t$.
However, it implies that $u$ takes the maximum inside $\Omega_t$, which contradicts the maximum principle.
In particular, we remark that $u(x) < t$ holds in $\Omega_t$.
\vskip 4mm
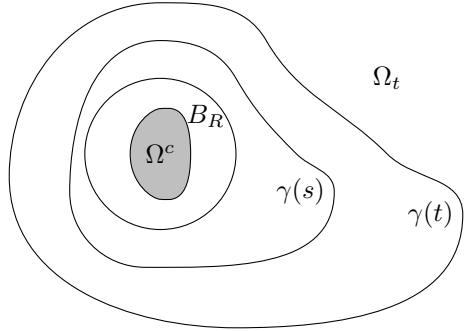
\begin{center}
\begin{tikzpicture}
\draw[fill = lightgray] (0.4,0) .. controls +(90:0.7) and +(0:0.1) .. (0,0.6)
			.. controls +(180:0.1) and +(90:0.4) .. (-0.4,0)
			.. controls +(270:0.4) and +(180:0.1) .. (0,-0.6)
			.. controls +(0:0.1) and +(270:0.7) .. (0.4,0);
\node at (0,0) {$\Omega^c$};

\draw (0,0) circle[radius=1];
\node at (0.6,0.5) {$B_R$};
\draw (0,1.5) to[out=0, in=120] ++(1, -0.5) to[out=300, in=135] ++(0.8, -1)
to[out =315, in=90] ++(0.5, -0.5) node[left] {$\gamma(s)$}
to[out=270, in=0] ++(-2.5, -1)
to[out=180, in=270] ++(-1,1) to[out=90, in=180] cycle; 
\draw (0,2) to[out=0, in=120] ++(1.4, -0.5) to[out=300, in=135] ++(1.6, -1.5)
to[out=315, in=90] ++(1,-0.8) node[left] {$\gamma(t)$}
to[out=270, in=0] ++(-3,-1.5) to[out=180, in=270] ++(-3,2) to[out=90, in=180] cycle;
\node at (3,1) {$\Omega_t$};
\end{tikzpicture}
\end{center} 
\vskip 4mm
It should be noted that there may exist a connected component of $u^{-1}(t)$ inside of $\gamma(t)$.  
Such a component causes a difficulty to make use of the coarea formula for $\nabla u$.   
The following proposition plays an important role for avoiding this difficulty.   
\begin{proposition}\label{pr:2.2} 
Let $u$ be a smooth solution to \eqref{eq:elliptic} satisfying \textbf{Assumption (S)}.
\par
\noindent
{\rm (i)} It holds that both $t$ and $t/2$ belong 
to $\mathcal{I}$ for almost every $t \in (0, t_\ast)$. 
\par
\noindent
{\rm (ii)} Let 
\begin{equation}
\label{eq:Itilde}
	\widetilde{\mathcal{I}} := \{t \in (0, t_\ast);\,\, \frac t2,\,\, t \in \mathcal{I} \}.   
\end{equation}
Suppose that $t \in \widetilde{\mathcal{I}}$.  
We set 
\[
	\tilde E_t := \{x \in \Omega_t\setminus \bar \Omega_{\frac t2}; \, 
	u(x) \in \left(\frac t2, t\right), \, \, \nabla u(x) \ne 0\}.  
\]
For $s \in (\frac t2, t) \cap \mathcal{I}$ we define 
\begin{align*}
	\omega_s&:=\{x \in \tilde E_t\cap \Int \gamma(s) ; \,\, u(x) < s\}, 
	\quad
	\omega:= \bigcup_{s\in(\frac t2, t)\cap \mathcal{I}}\omega_s, \\  
	E_t: &=\tilde E_t\setminus\omega =  
		\{x \in \Omega_t\setminus \bar \Omega_{\frac t2}; \, x \notin \omega, \, 
		u(x) \in \left(\frac t2, t\right), \, \, \nabla u(x) \ne 0\}.  
\end{align*}
Then there exists a null set 
$G_t \subset \Omega_t\setminus \bar \Omega_{\frac t2}$ such that 
\[
	E_t =  \bigcup_{s \in (\frac t2, t) \cap \mathcal{I}}\gamma(s) \cup G_t.   
\]
\end{proposition}
%%%%%%%%%%%%%%%%%%%%%%%%%%
\begin{remark}
\textup{(i)}
Roughly speaking,
$\tilde{E}_t$
is the region consisting of regular points of $u$ between $\gamma(t/2)$ and $\gamma(t)$.
$\omega_s$ is the region in $\Int \gamma(s)$ in which $u(x) < s$ holds,
and $\omega$ is the union of all such $\omega_s$ with $s \in (t/2,t)$.
Here, we remak that the maximum principle for $u(x)$ does not exclude
the possibility that $u(x)$ takes the positive minimum, and hence,
the region $\omega$ may exist.
The following figure shows an example that how $\omega$ appears.
\\
\textup{(ii)}
Later, we want to use the set
$\bigcup_{s \in (\frac t2, t) \cap \mathcal{I}}\gamma(s)$
as an integral region.
However, it is not clear whether this set is measurable or not.
Therefore, we define the region $E_t$ by subtracting $\omega$ from $\tilde{E}_t$.
By definition, $E_t$ is obviously Lebesgue measurable.
Moreover, the above proposition shows that
$E_t$ consists of the union of the set $\bigcup_{s \in (\frac t2, t) \cap \mathcal{I}}\gamma(s)$
and a set of zero Lebesgue measure.
\end{remark}
%%%%%%%%%%%%%%%%%%%%%%%%%%
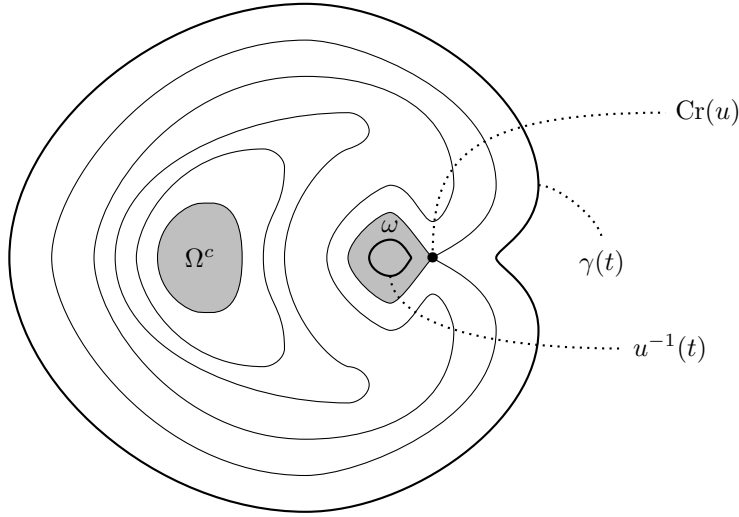
\begin{center}
\begin{tikzpicture}[every plot/.style={smooth cycle}, xscale=1.4, yscale=1.2]
%\draw[lightgray] (-2,-4) grid (4,4);

%Omega^c
\draw[fill = lightgray] (0.4,0) .. controls +(90:0.7) and +(0:0.1) .. (0,0.6)
			.. controls +(180:0.1) and +(90:0.4) .. (-0.4,0)
			.. controls +(270:0.4) and +(180:0.1) .. (0,-0.6)
			.. controls +(0:0.1) and +(270:0.7) .. (0.4,0);
\node at (0,0) {$\Omega^c$};

%level set 1
\draw (0.6, 0) .. controls + (90:0.5) and +(270:0.5) .. (0.8,1)
			.. controls +(90:0.1) and +(0:0.1) .. (0.6, 1.2)
			.. controls +(180:0.8) and +(90:0.4) .. (-0.6,0)
			.. controls +(270:0.4) and +(180:0.8) .. (0.6,-1.2)
			.. controls +(0:0.1) and +(270:0.1) .. (0.8,-1)
			.. controls +(90:0.5) and +(270:0.5) .. (0.6,0);
			
%\level set 2
\draw (0.8,0) .. controls +(90:0.6) and +(200:0.2) .. (1.4,1.2)
			.. controls +(20:0.1) and +(270:0.1) .. (1.6,1.4)
			.. controls +(90:0.1) and +(0:0.1) .. (1.4,1.6)
			.. controls +(180:1) and +(90:1) .. (-0.8,0)
			.. controls +(270:1) and +(180:1) .. (1.4,-1.6)
			.. controls +(0:0.1) and +(270:0.1) .. (1.6,-1.4)
			.. controls +(90:0.1) and +(0:0.1) .. (1.4,-1.2)
			.. controls +(160:0.2) and +(270:0.6) .. (0.8,0);
			
%level set 3
\draw (1.2,0) .. controls +(90:0.4) and +(180:0.1) .. (1.8,0.8)
			.. controls +(0:0.2) and +(180:0.1) .. (2.2,0.4)
			.. controls +(0:0.1) and +(270:0.1) .. (2.4,0.8)
			.. controls +(90:0.8) and +(0:0.8) .. (1,2)
			.. controls +(180:1.2) and +(90:0.7) .. (-1,0)
			.. controls +(270:0.7) and +(180:1.2) .. (1,-2)
			.. controls +(0:0.8) and +(270:0.8) .. (2.4,-0.8)
			.. controls +(90:0.1) and +(0:0.1) .. (2.2,-0.4)
			.. controls +(180:0.1) and +(0:0.2) .. (1.8,-0.8)
			.. controls +(180:0.1) and +(270:0.4) .. (1.2,0);
			
%level set 4
\draw[fill=lightgray] (2.2,0) .. controls +(135:0.1) and +(0:0.1) .. (1.8,0.5)
			.. controls +(180:0.1) and +(90:0.2) .. (1.4,0)
			.. controls +(270:0.2) and +(180:0.1) .. (1.8,-0.5)
			.. controls +(0:0.1) and +(225:0.1) .. (2.2,0);
\draw (2.2,0) .. controls +(45:0.1) and +(270:0.5) .. (2.8,0.8)
			.. controls +(90:1) and +(0:0.5) .. (1,2.4)
			.. controls +(180:1.4) and +(90:0.8) .. (-1.4,0)
			.. controls +(270:0.8) and +(180:1.4) .. (1,-2.4)
			.. controls +(0:0.5) and +(270:1) .. (2.8,-0.8)
			.. controls +(90:0.5) and +(315:0.1) .. (2.2,0);

%level set 5
\draw[thick] (2,0) .. controls +(120:0.1) and +(0:0.1) .. (1.8,0.2)
			.. controls +(180:0.1) and +(90:0.1) .. (1.6,0)
			.. controls +(270:0.1) and +(180:0.1) .. (1.8,-0.2)
			.. controls +(0:0.1) and +(240:.1) .. (2,0);
\draw[thick] (2.8,0) .. controls +(90:0.1) and +(270:0.5) .. (3.2,0.8)
			.. controls +(90:1) and +(0:1) .. (1,2.8)
			.. controls +(180:1.4) and +(90:1.4) .. (-1.8,0)
			.. controls +(270:1.4) and +(180:1.4) .. (1,-2.8)
			.. controls +(0:1) and +(270:1) .. (3.2,-0.8)
			.. controls +(90:0.5) and +(270:0.1) .. (2.8,0);

\node at (1.8,0.35) {$\omega$};

%u^{-1}(t)
\draw[dotted, thick] (1.8,-0.2) .. controls +(270:0.2) and +(180:2) .. (4,-1) node[right] {$u^{-1}(t)$};
\draw[dotted, thick] (3.2,0.8) .. controls +(0:0.2) and +(90:0.2) .. (3.8,0.2) node[below] {$\gamma(t)$};

%Cr(u)
\node at (2.2,0) {$\bullet$};
\draw[dotted, thick] (2.2,0) .. controls +(90:1) and +(180:2) .. (4.4,1.6) node[right] {$\mathrm{Cr}(u)$};

\end{tikzpicture}

\end{center}

{\it Proof of Proposition \ref{pr:2.2} .} (i) Defining the set 
$J :=\{t\in (0, t_\ast);\,\,\, t \in \mbox{Cr}(u)\,\,\,\mbox{or}\,\,\,  \frac{t}{2}\in \mbox{Cr}(u)\}$, we may prove that $|J| =0$.  Consider the set $K := (0, \frac{t_\ast}{2}) \cap \mbox{Cr}(u)$.  By the Morse-Sard theorem, it holds that $|K| =0$, which yields that 
$|2K| = |\{2s; \,\,\, s\in K\}| =0$.  
Since 
$$
J = \left((0,\,  t_\ast) \cap \mbox{Cr}(u)\right)\cup 2K, 
$$  
we have again by the Morse-Sard theorem that 
$$
|J| \le |(0,\,  t_\ast) \cap \mbox{Cr}(u)| + |2K| =0.  
$$
\par
(ii) The null set $G_t$ is indeed given by 
\[
	G_t = \{x \in \Omega_t\setminus \bar \Omega_{\frac t2}; \, x \notin \omega, \, 
	u(x) \in \left(\frac t2, t\right)\setminus\mathcal{I}, \, \, \nabla u(x) \ne 0\}.
\]
We first show that $|G_t| = 0$.
Let $x^0 \in G_t$, and put $t_0 = u(x^0)$. 
Since $\nabla u(x^0) \ne \mathbf{0}$, it follows from the implicit function 
theorem that there exists neighborhood $U$ if $x^0$, an open square 
$V=\{y=(y_1, y_2 )\in\mathbb{R}^2; x_1^0-\delta< y_1< x_1^0+\delta, \, 
t_0-\delta < y_2< t_0 + \delta\}$ with some $\delta>0$, 
and a $C^2$-diffeomorphism $\psi:U \to V$ such that 
\[
	\psi(x) \in (y_1-\delta, y_1 + \delta)\times \{ s \}
	\quad\mbox{for all $x \in u^{-1}( s )\in U$}. 
\]   
Defining $\phi=\psi^{-1}:V\to U$, we have by the Morse--Sard theorem that 
$|\phi^{-1}(G_t \cap U)|\le 2\delta|\mbox{Cr}(u)| =0$.   
Hence, it follows from the area formula 
(see, e.g., Evans-Gariepy \cite[Theorem 3.8]{EvGa} ) that 
\[
	0 = \int_{\phi^{-1}(G_t \cap U)}|\nabla \phi(y)|dy = \int_{G\cap U}dx 
	=|G\cap U|.  
\]
By the Lindel\"of covering theorem, the set $G_t$ is covered by the countably 
many open sets as above $U$ so that we may prove $|G_t| = 0$. 
\par    
We next show that 
\[
	E_t =  \bigcup_{s \in (\frac t2, t) \cap \mathcal{I}}\gamma(s) \cup G_t.   
\]
Obviously, it holds that
$ \bigcup_{s \in (\frac t2, t) \cap \mathcal{I}}\gamma(s) \cup G_t\subset E_t$,
and hence we may prove the converse inclusion 
relation.    
Let $x \in E_t$.  We have that $x \notin \omega$ 
and $s := u(x) \in (\frac t2, t)$ with $\nabla u(x)\ne \mathbf{0}$.  
Suppose that $x \notin G_t$, and we have $s \in \mathcal{I}$. 
Assume that $x \notin \gamma(s)$.  Then we see that $x$ is an element 
of another connected component of $u^{-1}(s)$ inside of $\gamma(s)$.  
Since $\nabla u(x) \ne \mathbf{0}$, 
it follows from the implicit function theorem that there is a 
neighborhood $U_x$ of $x$ such that every level set of $u$ contained 
in $U_x$ is a smooth curve.  Since the connected component containing 
$x$ is compact, it is covered by finitely many such neighborhoods $U_x$.  
Hence, there exists $\tilde s \in (\frac t2, t) \cap \mathcal{I}$ 
such that $s < \tilde s$ and $x \in \omega_{\tilde s} $.  
Since $\omega_{\tilde s} \subset \omega$, we have a contradiction.  
Hence it holds that $x \in \gamma(s)$ for 
$s \in (\frac t2, t) \cap \mathcal{I}$.        
This proves Proposition \ref{pr:2.2}.  \qed
\subsection{Key lemma}
The following is the key lemma to prove Theorem \ref{thm1}.

\begin{lemma}\label{lem2}
Let Assumptions (C) and (S) hold.  
Assume that $R$, $t_\ast$ and $\mathcal{I}$ are the same as in Subsection 2.1.  
Let $\tilde{\mathcal {I}}$ be the set defined in Proposition \ref{pr:2.2}.    
For every $p \in (0,\infty)$,
there exists a constant
$C = C(R,a_{ij},\mathbf{b},c,p) > 0$ such that
\begin{equation}\label{eq:thm1}
	u(x) \le C |x|^{-\frac{2}{p}} \left( \int_{E_t } |u(y)|^p \,dy \right)^{1/p}
\end{equation}
holds for all $t \in \widetilde{\mathcal{I}}$ 
and all $x \in \gamma(t)$. 
\end{lemma}
%%%%%%%%%%%%%%%%%%
{\it Proof.}
We shall use the same notations as those in the previous subsection.
Let $t \in \widetilde{\mathcal{I}}$ and let $\rho > R$ be a sufficiently large parameter satisfying
$\partial \Omega_t = \gamma(t) \subset B_{\rho}$.
Define a cut-off function $\eta_{\rho}(x)$ by
\[
	\eta_{\rho} (x) := \eta \left( \frac{x}{\rho} \right),
	\quad \text{where} \quad
	\eta \in C_0^{\infty}(\mathbb{R}^2), \ 0 \le \eta \le 1, \ 
	\eta(x) = \begin{cases} 1 &(|x| \le 1),\\ 0 &(|x| > 2). \end{cases}
\]
By the definition, it is easy to see that 
$$
	| \nabla^k \eta_{\rho}(x) | \le C \rho^{-k}, 
	\quad
	\supp\nabla^k \eta_{\rho} \subset\{x \in \re^2; \rho< |x| < 2\rho\} \quad (k=1,2).
$$
Multiplying the equation \eqref{eq:elliptic} by $\eta_{\rho}$,
and then integrating the resulting identity over $\Omega_t$
and integrating by parts twice,
we obtain
\begin{align}\label{eq:int:eq}
	0 &=
	\int_{\Omega_t} L u (x) \eta_{\rho}(x) \,dx \\
	&= - \int_{\gamma(t)} \sum_{i,j=1}^2 \frac{\partial_i u}{|\nabla u|} a_{ij}(x)\partial_j u(x) \eta_{\rho}(x)\,dS \\
	&\quad
	+ \int_{\gamma(t)} \sum_{i,j=1}^2 \frac{\partial_j u}{|\nabla u|} a_{ij}(x) u(x) \partial_i \eta_{\rho}(x) \,dS 
	\\
	&\quad
	- \int_{\Omega_t} \sum_{i,j=1}^2  \partial_j ( a_{ij}(x) \partial_i \eta_{\rho}(x) ) u(x) \,dx
	\\
	&\quad
	+ \int_{\gamma(t)} \frac{\nabla u}{|\nabla u|} \cdot \mathbf{b}(x) u (x) \eta_{\rho}(x) \,dS
	- \int_{\Omega_t} u(x) \mathbf{b}(x) \cdot \nabla \eta_{\rho}(x) \,dx
	\\
	&\quad
	+ \int_{\Omega_{t}} (-\diver \mathbf{b}(x) + c(x)) u (x) \eta_{\rho}(x) \,dx,
\end{align}
where we remark that the unit outward normal vector on
$\partial \Omega_t = \gamma(t)$
is
$\nabla u/ |\nabla u|$,
and the second term of the right-hand side vanishes thanks to
$\nabla \eta_{\rho} = 0$ on $\gamma(t)$.
First, by the assumption $\mathbf{(C4)}$
and the fact that $0 < u(x) < t$ in $\Omega_t$, we see that 
the last term is estimated as
\[
	 \int_{\Omega_{t}} (-\diver \mathbf{b}(x) + c(x)) u (x) \eta_{\rho}(x) \,dx
	 \le t \| ( \diver \mathbf{b} - c )_{-} \|_{L^1(\Omega)}.
\]
Next, for the second and fourth terms of the right-hand side of \eqref{eq:int:eq},
by the assumption $\mathbf{(C2)}$ and again by the fact that $0 < u(x) < t$  in $\Omega_t$, 
we have
\begin{align*}
	&- \int_{\Omega_t} \sum_{i,j=1}^2  \partial_j ( a_{ij}(x) \partial_i \eta_{\rho}(x) ) u(x) \,dx \\
	&\quad
	\le
	t \int_{B_{2\rho} \setminus B_{\rho}} \sum_{i,j=1}^2
		\left| \partial_j a_{ij}(x) \partial_i \eta_{\rho}(x) + a_{ij}(x) \partial_i \partial_j \eta_{\rho}(x) \right| \,dx \\
	&\quad
	\le Ct \int_{B_{2\rho} \setminus B_{\rho}} \sum_{i,j=1}^2 
		(|\partial_j a_{ij}(x)|\rho^{-1} + |a_{ij}(x)|\rho^{-2})\, dx \\
	&\quad
	\le C t
\end{align*}
and
\begin{align*}
	- \int_{\Omega_t} u(x) \mathbf{b}(x) \cdot \nabla \eta_{\rho}(x) \,dx
	&\le
	t \int_{B_{2\rho}\setminus B_{\rho}} | \mathbf{b}(x) \cdot \nabla \eta_{\rho}(x) | \,dx \\
	& \le
	Ct\rho^{-1}\int_{B_{2\rho}\setminus B_{\rho}} | \mathbf{b}(x)|dx \\
	&\le 
	C t.
\end{align*}
Finally, the third term of the right-hand side is calculated as
\begin{align*}
	\int_{\gamma(t)} \frac{\nabla u}{|\nabla u|} \cdot \mathbf{b}(x) u (x) \eta_{\rho}(x) \,dS
	&=
	t \int_{\gamma(t)} \frac{\nabla u}{|\nabla u|} \cdot \mathbf{b}(x) \,dS \\
	&=
	-t \left[ \int_{\Omega \setminus \Omega_t} \diver \mathbf{b}(x) \,dx
	 		- \int_{\partial \Omega} \mathbf{n}\cdot \mathbf{b}(x) \,dS
	\right] \\
	&\le
	t \| (\diver \mathbf{b})_- \|_{L^1(\Omega)} + C \| \mathbf{b} \|_{L^{\infty}(\partial\Omega)} t .
\end{align*}
It should be noted by Assumptions \textbf{(C3)} and \textbf{(C4)} that 
$0\le (\diver \mathbf{b})_{-}\le (\diver \mathbf{b} - c)_- $, which yields that 
$(\diver \mathbf{b})_{-} \in L^1(\Omega)$.    
Putting the above estimates together to \eqref{eq:int:eq}, we conclude
\begin{align}
	\int_{\gamma(t)} \frac{1}{|\nabla u|} \sum_{i,j=1}^2 a_{ij}(x) \partial_i u \partial_j u \,dS
	\le
	C(a_{ij}, \mathbf{b}, c) t
\end{align}
with some constant $C(a_{ij}, \mathbf{b}, c) >0$.
Furthermore, by  $\mathbf{(C1)}$ and the above estimate, we have
\begin{align}
	\int_{\gamma(t)}  |\nabla u|  \,dS
	&\le \lambda^{-1} \int_{\gamma(t)} \frac{1}{|\nabla u|} \sum_{i,j=1}^2 a_{ij}(x) \partial_i u \partial_j u \,dS 
	\le C_* t
\end{align}
with some constant
$C_* = C_*(a_{ij},\mathbf{b},c, \lambda) > 0$. 
\par
For $t \in \widetilde{\mathcal{I}}$, it follows from Proposition \ref{pr:2.2} and 
the coarea formula (see, e.g.,  \cite{EvGa}) that 
\begin{equation}\label{eq:coarea}
	\int_{E_t} f |\nabla u| \,dx = \int_{t/2}^{t} \left( \int_{\gamma(\tau)} f \,dS \right) \,d\tau
\end{equation}
for all $f \in C(\Omega)$.  
Taking $f = |\nabla u|$ in (\ref{eq:coarea}), we have 
\begin{align}
	\int_{E_t} |\nabla u|^2 \,dx
	&= \int_{t/2}^{t} \left( \int_{\gamma(\tau)} |\nabla u| \,dS \right) \,d\tau \\
	&\le \int_{t/2}^{t} C_* \tau \,d\tau \\
	&\le C_*t^2.
\end{align}
Furthermore, application of \eqref{eq:coarea} to $f =1$ enables us to obtain 
\begin{align}
	\int_{t/2}^{t} \left( \int_{\gamma(\tau)} 1 \,dS \right) \,d\tau
	&=
		\int_{E_t} |\nabla u| \,dx \\
	&\le
		|E_t|^{1/2} \left( \int_{E_t} |\nabla u|^2 \,dx \right)^{1/2} \\
	&\le
		\left( C_*|E_t| t^2 \right)^{1/2} \\
	&\le \left( 2^pC_*  t^{2-p} \int_{E_t} |u(x)|^p \,dx \right)^{1/2}.
\end{align}
From the above estimate,
we have that for every $t \in \widetilde{\mathcal{I}}$, there exists $\tau \in [t/2,t]$ such that
\begin{equation}\label{eq:est:lvset}
	\frac{t}{2} \mathcal{H}^{1} (\gamma(\tau) ) \le \left( 2^pC_*  t^{2-p} \int_{E_t} |u(x)|^p \,dx \right)^{1/2}.
\end{equation}
In fact, if
\[
	\frac{t}{2} \mathcal{H}^{1} (\gamma(\tau) ) > \left( 2^pC_*  t^{2-p} \int_{E_t} |u(x)|^p \,dx \right)^{1/2}
\]
holds for all $\tau \in [t/2,t]$, 
by integration of both sides over $[t/2,t]$ with respect to $\tau$, we have
\[
	\frac{t}{2} \int_{t/2}^t \left( \int_{\gamma(\tau)} \, dS \right) d\tau 
	> \frac{t}{2} \left( 2^p C_*  t^{2-p} \int_{E_t} |u(x)|^p \,dx \right)^{1/2},
\]
which contradics the previous inequality. Therefore, we obtain \eqref{eq:est:lvset}.
Let
\[
	g(t) := \sup \{ |x| ; x \in \gamma(t) \}.
\]
Then,
we see that
$2g(t) \le \mathcal{H}^1(\gamma(\tau))$
with $\tau$ satisfying \eqref{eq:est:lvset}
by an elementary geometric argument.
Indeed, let $x_{\ast} \in \gamma(t)$ be a point satisfying $g(t) = |x_{\ast}| = \max\{ |x| ; x\in \gamma(t) \}$.
Since $\tau \in [t/2, t]$, we have $x_{\ast} \in \overline{\Int \gamma(\tau)}$.
Let $\ell$ be any line segment extending $[0,x_{\ast}]$ such that its endpoints lie on $\gamma(\tau)$
with denoting its length by $|\ell|$.
%We extend the line segment $[0, x_{\ast}]$ to a line segment whose endpoints are on $\gamma(\tau)$,
%and we call it $\ell$ and denote its length by $|\ell|$.
Then, it is obvious that $2|\ell|$ does not exceed $\mathcal{H}^1(\gamma(\tau))$,
and this immediately implies $2g(t) \le \mathcal{H}^1(\gamma(\tau))$.
It should be noticed that this argument uses that the space dimension is two.
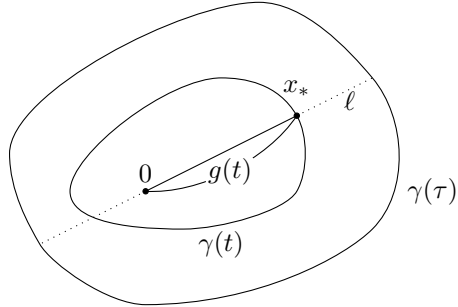
\footnote{If the space dimension is three or higher, this geometric argument does not work well.}
%%%%%%%%%%%%%%%%%%
\begin{center}
\begin{tikzpicture}

%Omega^c
\filldraw (0,0) circle[radius=1pt];
\node at (0,0) [above] {$0$};

\draw[bend right,distance=0.7cm] (0,0) to node[fill=white,inner sep=0.2pt,circle] {$g(t)$} (2,1);
\draw (0,0) -- (2,1) .. controls +(120:0.5) and +(0: 0.3) .. (1, 1.5)
				.. controls +(180:0.7) and +(90:0.5) .. (-1,0)
				.. controls +(270:0.5) and +(180:0.3) .. (0.5, -0.5)
				.. controls +(0:0.5) and +(240:0.3) .. (2, 0)
				.. controls +(60:0.3) and +(300:0.3) .. (2,1);
\filldraw (2,1) circle[radius=1pt] node[above=4pt] {$x_{\ast}$};
\node at (1,-0.7) {$\gamma(t)$};

\draw[dotted] (2,1) -- (3,1.5);
\draw[dotted] (0,0) -- (-1.4,-0.7);

\draw (3,1.5) .. controls +(130: 0.5) and +(0:0.6) .. (1.5, 2.5)
			.. controls +(180:1 ) and +(90:0.8) .. (-1.8, 0.5)
			.. controls +(270:0.5) and +(120: 0.3) .. (-1.4, -0.7)
			.. controls +(300:0.3) and +(180:0.5) .. (0,-1.5)
			.. controls +(0:1.5) and +(225: 0.5) .. (3,-0.5)
			.. controls +(45:0.7) and +(310:0.7) .. (3,1.5);
\node at (2.7,1.2) {$\ell$};

\node at (3.8,0) {$\gamma(\tau)$};

\end{tikzpicture}

\end{center}
%%%%%%%%%%%%%%%%%
\par
Therefore, we conclude
\[
	t g(t) \le
	\frac{t}{2} \mathcal{H}^1(\gamma(\tau))
	\le
	\left( 2^p C_*  t^{2-p} \int_{E_t} |u(x)|^p \,dx \right)^{1/2},
\]
which implies
\[
	t g(t)^{\frac{2}{p}} \le C_{*}^{\prime} \left( \int_{E_t} |u(x)|^p \,dx \right)^{1/p}.
\]
Thus, for every $t\in \widetilde{\mathcal{I}}$ and $x \in \gamma(t)$, we have
\[
	u(x) \le C_{*}^{\prime} |x|^{-\frac{2}{p}} \left( \int_{E_t} |u(x)|^p \,dx \right)^{1/p}.
\]
This completes the proof of Lemma \ref{lem2}. 
\qed
\subsection{Proof of Theorem \ref{thm1}} 
The proof of Theorem \ref{thm1} is based on Lemma \ref{lem2}.
(i) Since $u \in L^{p, q}(\Omega)$ for $1 \le  p, q < \infty$, 
for every $\varepsilon >0 $ there is $t_\ep$ such that 
\begin{equation}\label{eqn:2.4}
\left( \int_0^{s} \tau^q \left( | \{ x \in \Omega ;\, |u(x)| > \tau \} | \right)^{q/p} \,\frac{d\tau}{\tau} \right)^{1/q} < \varepsilon
\quad 
\mbox{for all $0 < s \le t_\varepsilon $}.   
\end{equation} 
Furthermore, since $\lim_{|x|\to\infty}u(x) = 0$, 
there exists $R_\varepsilon > R$ such that 
\begin{equation}\label{eqn:2.5} 
u(x) < t_\varepsilon 
\quad
\mbox{for all $x\in \mathbb{R}^2$ with $|x| \ge R_\varepsilon$}.  
\end{equation}
Let $|x| \ge R_\varepsilon$.  
Defining $t := u(x)$,  we have by \eqref{eqn:2.5} that $t < t_\varepsilon$.  
We need to consider the following four cases
\begin{itemize}
\item[Case 1.] $x \in \gamma(t)$ for $t\in \widetilde{\mathcal{I}}$; 
\item[Case 2.] $x \in \gamma(t)$ for $t \in \mathcal{I}\setminus\widetilde{\mathcal{I}}$;
\item[Case 3.] $x \in u^{-1}(t)\setminus \gamma(t)$ for $t \in \mathcal{I}$; 
\item[Case 4.] $x \in u^{-1}(t)$ for $t \in \mathrm{Cr}(u)$. 
\end{itemize}
Here, we recall that
$\mathcal{I}$ and $\widetilde{\mathcal{I}}$ are defined by
\eqref{eq:I} and \eqref{eq:Itilde}, respectively.
\par
{\it Case 1.}  
Since we have by Proposition \ref{pr:2.2} that 
$E_t =  \bigcup_{s \in (\frac t2, t) \cap \mathcal{I}}\gamma(s) \cup G$ with $|G| =0$, 
 it follows from \eqref{eqn:2.4}, \eqref{eqn:2.5} and Lemma \ref{lem2} that  
\begin{align}\label{eqn:2.6}
u(x) &\le C|x|^{-\frac 2p}\left(\int_{E_t}|u(y)|^pdy\right)^{\frac1p}\nonumber  \\
&= C|x|^{-\frac 2p}\left(\int_{\bigcup_{s \in (\frac t2, t) \cap \mathcal{I}}\gamma(s)}|u(y)|^pdy\right)^{\frac1p} \nonumber \\
&\le  C|x|^{-\frac 2p} t|\{y \in B_R^c; |u(y)|>t/2\}|^{\frac1p} \nonumber \\
&\le C|x|^{-\frac 2p}\left(\int_0^{t/2}\tau^q\frac {d\tau}{\tau}\right)^{\frac1q}|\{y \in B_R^c; |u(y)|>t/2\}|^{\frac1p}
 \nonumber \\
&\le C|x|^{-\frac 2p}
\left(\int_0^{t/2}\tau^q |\{y \in B_R^c; |u(y)|>\tau\}|^{\frac qp} 
\frac {d\tau}{\tau}\right)^{\frac1q} \nonumber \\
&\le C\varepsilon |x|^{-\frac 2p}. 
\end{align} 
\par
{\it Case 2.} Since $\nabla u(x) \ne \mathbf{0}$, there are two sequences 
$\{t_j\}_{j=1}^\infty$ and $\{x_j\}_{j=1}^\infty$ 
such that 
\begin{align}
& t_j \in \widetilde{\mathcal{I}}, 
\quad j=1, \dots
\quad
\mbox{with $\lim_{j\to\infty}t_j = t$}, \label{eqn:2.7} \\ 
& 
x_j \in \gamma(t_j) \quad j=1, \dots
\quad
\mbox{with $\lim_{j\to\infty}x_j = x$}. \label{eqn:2.8}
\end{align}
For validity of \eqref{eqn:2.8}, we may show that if a connected component 
of $u^{-1}(t_j)$ is sufficiently close to $\gamma(t)$, then it coincides with 
$\gamma(t_j)$.  Indeed, since $t \in \mathcal{I}\setminus\widetilde{\mathcal{I}}$,  
for every $y \in \gamma(t)$ we have $\nabla u(y) \ne \mathbf{0}$, and hence 
there exists a neighborhood $U_y$ of $y$ such that every level set of $u$ in 
$U_y$ consists of a smooth curve.  Since $\gamma(t)$ is compact, there are 
finitely many points $y_1, \dots, y_N$ of $\gamma(t)$ such that $\gamma(t) \subset \bigcup_{k=1}^NU_{y_k}$.   
Since the connected component of $u^{-1}(s)$ for $s \in \widetilde{\mathcal{I}}$ 
contained in $\bigcup_{k=1}^NU_{y_k}$ is a closed curve 
which is homotopic to $\gamma(t)$ and which contains $B_R$ in its inside, 
it necessarily coincides with $\gamma(s)$.  
Hence, we may choose $x_j \in \gamma(t_j)$ in such a way that the condition \eqref{eqn:2.8} is fulfilled.  
\par
Applying Lemma \ref{lem2} to $\{t_j\}_{j=1}^\infty$ and $\{x_j\}_{j=1}^\infty$, we have 
that 
$$
u(x_j) \le C|x_j|^{-\frac 2p}\left(\int_{E_{t_j}}|u(y)|^pdy\right)^{\frac 1p}, 
\quad
j=1, 2, \dots.  
$$ 
Letting $j\to \infty$ in both sides of the above estimate, similarly to \eqref{eqn:2.6}, 
we have by \eqref{eqn:2.4} that 
\begin{align}
u(x) &\le C|x|^{-\frac 2p}\limsup_{j\to\infty}\left(\int_{E_{t_j}}|u(y)|^pdy\right)^{\frac 1p} 
\nonumber \\
&\le  C|x|^{-\frac 2p}\limsup_{j\to\infty}
\left(\int_0^{t_j/2}\tau^q|\{y \in B_R^c; |u(y)|>\tau\}|^{\frac qp}\frac {d\tau}{\tau}\right)^{\frac1q}\nonumber \\
&\le C\varepsilon|x|^{-\frac 2p}.  \label{eqn:2.9}
\end{align}
\par
{\it Case 3.} 
In this case, we see that $x$ is contained in some connected component of $u^{-1}(t)$ which 
is different from $\gamma(t)$.  
Taking the half line $l_x$ connecting the origin in $\mathbb{R}^2$ and $x$, 
we take $x_\ast \in \gamma(t)$ on $l_x$ in such a way that the distance from the origin to 
$x_\ast$ the longest among intersection points of $\gamma(t)$ and $l_x$.   
Since $x_\ast \in \gamma(t)$ with $|x| \le |x_\ast|$ and $t \in \mathcal{I}$, 
we may apply the estimate \eqref{eqn:2.9} with $x$ replaced by $x_\ast$ to obtain that 
\begin{align}
u(x) = t = u(x_\ast) &\le 
C|x_\ast|^{-\frac 2p}\limsup_{j\to\infty}\left(\int_{E_{t_j}}|u(x)|^pdx\right)^{\frac 1p} \nonumber  \\
&\le   C\ep|x|^{-\frac 2p},  \label{eqn:2.10}
\end{align}
where $\{t_j\}_{j=1}^\infty$ may be chosen like \eqref{eqn:2.7} with $x$ replaced by $x_\ast$ 
in \eqref{eqn:2.8}.  
It should be noted that if $x \in u^{-1}(t)\setminus \gamma(t)$ with $t \in \widetilde{\mathcal{I}}$, 
then we have by taking $x_\ast\in \gamma(t)$ as in Case 1 that 
\begin{align}  
u(x) = t = u(x_\ast) &\le C|x_\ast|^{-\frac 2p}\left(\int_{E_t}|u(y)|^pdy\right)^{\frac 1p} 
\nonumber \\
&\le  C|x|^{-\frac 2p}\left(\int_{E_t}|u(y)|^pdy\right)^{\frac 1p}.  \label{eqn:2.11}
\end{align} 
\par
\bigskip
\begin{center}
\begin{tikzpicture}

%Omega^c
\draw[fill = lightgray] (0.4,0) .. controls +(90:0.7) and +(0:0.1) .. (0,0.6)
			.. controls +(180:0.1) and +(90:0.4) .. (-0.4,0)
			.. controls +(270:0.4) and +(180:0.1) .. (0,-0.6)
			.. controls +(0:0.1) and +(270:0.7) .. (0.4,0);
\node at (-0.1,0.8) {$\Omega^c$};

\filldraw (0,0) circle[radius=1pt];
\node at (0,0) [left] {$0$};

%gamma(t)
\draw (3, 0) .. controls +(90:0.7) and +(330:0.5) .. (2,1.5)
			.. controls +(150:0.3) and +(0:0.8) .. (0,2)
			.. controls +(180:0.8) and +(90:0.5) .. (-2,0)
			.. controls +(270:0.8) and +(150:0.3) .. (-1,-1)
			.. controls +(330:2) and + (270:0.5) ..(3,0);
\node at (2.7,-0.8) {$\gamma(t)$};

%x
\filldraw (1.5,0) circle[radius=1pt] node[below] {$x$};

%u^{-1}(t)
\draw (1.5, 0) .. controls +(0:0.2) and +(270:0.1) .. (2,0.4)
			.. controls +(90:0.3) and +(0:0.2) .. (1.3, 0.8)
			.. controls +(180:0.2) and +(90:0.3) .. (1,0.4)
			.. controls +(270:0.3) and +(180:0.3) .. (1.5,0);
\node at (1.2,1.1) {$u^{-1}(t)$};

\draw[dotted, thick] (0,0)--(3,0);
\filldraw (3,0) circle[radius=1pt] node[right] {$x_{\ast}$};

\end{tikzpicture}

\end{center}
\par
{\it Case 4. }
Since $u^{-1}(t)$ is compact in $\mathbb{R}^2$, there exists $x^\ast \in u^{-1}(t)$ such that 
$|x^\ast|= \max\{|y|;\,\, y\in u^{-1}(t)\}$.  Let $\nu \equiv x^\ast/|x^\ast|$.  
It holds that $\lim_{\theta \to 0}u(x^\ast + \theta\nu) = u(x^\ast) = t$.  
By the definition of $x^\ast$, we have $t_\nu \equiv u(x^\ast + \nu) \neq t$.
Let us assume $t_{\nu} < t$ for simplicity.
The other case may be handled similarly.
Hence, there exist two sequences $\{t_j\}_{j=1}^\infty$ and $\{x_j\}_{j=1}^\infty$   
such that        
\begin{align}
	& t_j \in \widetilde{\mathcal{I}} \cap (t_{\nu}, t), 
	\quad j=1, \dots
	\quad
	\mbox{with $\lim_{j\to\infty}t_j = t$}, \label{eqn:2.12} \\ 
	& 
	x_j \in u^{-1}(t_j) \quad j=1, \dots
	\quad
	\mbox{with $\lim_{j\to\infty}x_j = x^\ast$}. \label{eqn:2.13}
\end{align}
Indeed,
we first note that the intermediate value theorem applied to the function
$[0,1] \ni \theta \mapsto u(x^{\ast}+\theta \nu)$
implies
$[x^\ast, x^\ast + \nu] \cap u^{-1}(t_j) \neq \emptyset$ for each $t_j \in \widetilde{\mathcal{I}} \cap \,(t_{\nu}, t).$
Then, we may choose $x_j$ the nearest point
to $x^{\ast}$
on 
$[x^\ast, x^\ast + \nu] \cap u^{-1}(t_j)$.   
For each $x_j$ we may apply \eqref{eqn:2.11} to obtain 
$$
u(x_j) \le C|x_j|^{-\frac 2p}\left(\int_{E_{t_j}}|u(y)|^pdy\right)^{\frac1p}.
$$  
Letting $j \to \infty$ in the above, we have by \eqref{eqn:2.4} that 
\begin{align}
u(x) = t = u(x^\ast)
&\le 
C|x^\ast|^{-\frac2p}\limsup_{j\to \infty}\left(\int_{E_{t_j}}|u(y)|^pdy\right)^{\frac1p} \nonumber \\
&\le 
 C\varepsilon|x|^{-\frac2p}.   \label{eqn:2.13}
\end{align}
Since $\varepsilon >0$ is arbitrary,  it follows from \eqref{eqn:2.6}, \eqref{eqn:2.9}, 
\eqref{eqn:2.10} and \eqref{eqn:2.13} that 
$$
u(x) = o(|x|^{-\frac 2p}) 
\quad
\mbox{as $|x| \to  \infty$}.  
$$
\par
(ii) Since $u \in L^{p, \infty}(\Omega)$,  instead of \eqref{eqn:2.4} it holds that 
$$
\|u\|_{L^{p, \infty}}=\sup_{s>0}s|\{y\in B_R^c; |u(y)|>s\}|^{\frac1p} < \infty.  
$$
Hence, in Case 1, we have similarly to \eqref{eqn:2.6} that 
\begin{equation}\label{eqn:2.14}
u(x) \le C|x|^{-\frac 2p}t|\{y \in B_R^c; |u(y)|>t/2\}|^{\frac1p} \le C|x|^{-\frac 2p}\|u\|_{L^{p, \infty}}.   
\end{equation} 
Similarly to \eqref{eqn:2.9}, \eqref{eqn:2.10} and \eqref{eqn:2.13}, we obtain the 
same estimate as \eqref{eqn:2.14} in other Cases 2, 3 and 4.  
This completes the proof of Theorem \ref{thm1}.   

\section*{Acknowledgement}
This work was supported by JSPS 	
Grant-in-Aid for Scientific Research (A) JP21H04433, 
Research (A) JP22H00097, Research (C) JP24K06811
and Grant-in-Aid for Transformative Research Areas (B) 25H01453.  
The authors are grateful to the anonymous referees for their helpful comments and suggestions.

%\par
%The authors declare no conflicts of interest.
%Data sharing is not applicable to this article as no data were created or analyzed in this study.
\par
\bigskip
{\bf Declaration of interests} The authors declare no conflicts of interest.
\par
{\bf Data availability statement} 
Data sharing is not applicable to this article as no data were created or analyzed in this study.

\end{document}